\documentclass[a4paper]{amsart}
\usepackage[english]{babel}
\usepackage[utf8]{inputenc}
\usepackage[T1]{fontenc}
\usepackage{latexsym}
\usepackage{amsmath}
\usepackage{amscd}
\usepackage{amsfonts}
\usepackage{amsthm}
\usepackage{amssymb}
\usepackage[dvips]{pstricks}
\usepackage{amsaddr}
\usepackage{mathrsfs}
\usepackage{textcomp}
\usepackage{xfrac}
\usepackage{hyperref}

\numberwithin{equation}{section}
\newtheorem{thm}{Theorem}

\newtheorem{proposition}{Proposition}[section]
\newtheorem{corollary}[proposition]{Corollary}
\newtheorem{lemma}[proposition]{Lemma}

\begin{document}
\author{Titus Lupu}
\email{titus.lupu@eth-its.ethz.ch, titus.lupu@math.u-psud.fr}
\title{Loop percolation on discrete half-plane}

\newenvironment{proofof1}{\paragraph{\textit{Proof of the Theorem \ref{ThmCcritLeq1}.}}}{\hfill$\square$}
\newenvironment{proofof2}{\paragraph{\textit{Proof of the Theorem \ref{ThmIntermVert}.}}}{\hfill$\square$}
\newenvironment{proofof3}{\paragraph{\textit{Proof of the Theorem \ref{ThmCutOff}.}}}{\hfill$\square$}
\newenvironment{proofof4}{\paragraph{\textit{Proof of the Theorem \ref{ThmCutOff}.}}}{\hfill$\square$}
\newenvironment{proofofprop}{\paragraph{\textit{Proof of the Proposition \ref{PropStrictIneq}.}}}{\hfill$\square$}

\begin{abstract}
We consider the random walk loop soup on the discrete half-plane 
$\mathbb{Z}\times\mathbb{N}^{\ast}$ and study the percolation problem, i.e. the existence of an infinite cluster of loops. We show that the critical value of the intensity is equal to $\frac{1}{2}$. The absence of percolation at intensity $\frac{1}{2}$ was shown in a previous work. We also show that in the supercritical regime, one can keep only the loops up to some large enough upper bound on the diameter and still have percolation.
\end{abstract}

\maketitle

\section{Introduction}
\label{SecInro}

We will consider discrete (rooted) loops on $\mathbb{Z}^{2}$, that is to say finite paths to the nearest neighbours on $\mathbb{Z}^{2}$ that return to the origin and visit at least two vertices. The rooted random walk loop measure $\mu_{\mathbb{Z}^{2}}$ gives to each rooted loop of lengths $2n$ the mass $(2n)^{-1}4^{-2n}$. It was introduced in \cite{LawlerFerreras2007RWLoopSoup}. In \cite{LeJan2011Loops} are considered loops parametrised by continuous time rather than discrete time. $\mu_{\mathbb{Z}^{2}}$ has a continuous analogue, the measure $\mu_{\mathbb{C}}$ on the Brownian loops on $\mathbb{C}$. Let $\mathbb{P}^{t}_{z,z'}(\cdot)$ be the standard Brownian bridge probability measure from $z$ to $z'$ of length $t$. 
$\mu_{\mathbb{C}}$ is a measure on continuous time-parametrised loops on $\mathbb{C}$ defined as
\begin{displaymath}
\mu_{\mathbb{C}}(\cdot):=\int_{\mathbb{C}}\int_{t>0}
\mathbb{P}^{t}_{z,z}(\cdot)\dfrac{dt}{2\pi t^{2}}
\dfrac{d\bar{z}\wedge dz}{2i},
\end{displaymath}
where $\frac{d\bar{z}\wedge dz}{2i}$ is the standard volume form on $\mathbb{C}$. The measure $\mu_{\mathbb{C}}$ was introduced in \cite{LawlerWerner2004ConformalLoopSoup}.

Given $\alpha>0$ we will denote by $\mathcal{L}^{\mathbb{Z}^{2}}_{\alpha}$ respectively $\mathcal{L}^{\mathbb{C}}_{\alpha}$ the Poisson ensemble of intensity $\alpha\mu_{\mathbb{Z}^{2}}$ respectively $\alpha\mu_{\mathbb{C}}$, called random walk respectively Brownian loop soup. In \cite{LawlerFerreras2007RWLoopSoup} it was shown that one can approximate $\mathcal{L}^{\mathbb{C}}_{\alpha}$ by a rescaled version of
$\mathcal{L}^{\mathbb{Z}^{2}}_{\alpha}$. If $A$ is a subset of $\mathbb{Z}^{2}$ we will denote by $\mathcal{L}^{A}_{\alpha}$ the subset of 
$\mathcal{L}^{\mathbb{Z}^{2}}_{\alpha}$ made of loops contained in $A$. If $U$ is an open subset of $\mathbb{C}$ we will denote by $\mathcal{L}^{U}_{\alpha}$ the subset of $\mathcal{L}^{\mathbb{C}}_{\alpha}$ made of loops contained in $U$. For $\delta >0$ we will denote by $\mathcal{L}^{A,\geq\delta}_{\alpha}$ respectively 
$\mathcal{L}^{U,\geq\delta}_{\alpha}$ the subset of random walk loops $\mathcal{L}^{A}_{\alpha}$ respectively Brownian loops $\mathcal{L}^{U}_{\alpha}$
made of loops of diameter greater or equal to $\delta$. Similarly we will use the notation $\mathcal{L}^{A,\leq\delta}_{\alpha}$ for the loops of diameter smaller or equal to $\delta$.

We will consider clusters of loops. Two loops $\gamma$  and $\gamma'$ in a Poisson ensemble of discrete or continuous loops belong to the same cluster if there is a chain of loops 
$\gamma_{0},\gamma_{1},\dots,\gamma_{n}$ in this Poisson ensemble such that $\gamma_{0}=\gamma$, $\gamma_{n}=\gamma'$ and $\gamma_{i}$ and $\gamma_{i-1}$ visit a common point. For all $\alpha>0$, loops in $\mathcal{L}^{\mathbb{Z}^{2}}_{\alpha}$ as well as in $\mathcal{L}^{\mathbb{C}}_{\alpha}$ form a single cluster. Thus we will consider loops on discrete half-plane $\mathsf{H}=\mathbb{Z}\times\mathbb{N}^{\ast}$ and on continuous half-plane $\mathbb{H}=\lbrace z\in\mathbb{C}\vert \Im(z)>0\rbrace$, mainly from the angle of existence of an unbounded cluster. 

The percolation problem for Brownian loops was studied in \cite{SheffieldWerner2012CLE}. It was shown that there is a critical intensity $\alpha_{\ast}^{\mathbb{H}}\in (0,+\infty)$ such that for $\alpha\in (0,\alpha_{\ast}^{\mathbb{H}}]$, $\mathcal{L}^{\mathbb{H}}_{\alpha}$ has only bounded clusters, and for $\alpha>\alpha_{\ast}^{\mathbb{H}}$ the loops in $\mathcal{L}^{\mathbb{H}}_{\alpha_{\ast}^{\mathbb{H}}}$ form one single cluster. The critical intensity was identified to be equal to $1$. But actually $\alpha_{\ast}^{\mathbb{H}}=\frac{1}{2}$. In \cite{SheffieldWerner2012CLE} the outer boundaries of outermost clusters in a sub-critical Brownian loop soup were identified to be a Conformal Loop Ensemble $CLE_{\kappa}$ with the following relation between $\alpha$ and $\kappa$.
\begin{displaymath}
\alpha=\dfrac{(3\kappa-8)(6-\kappa)}{2\kappa}.
\end{displaymath}
The critical value of $\kappa$ corresponds to $CLE_{4}$. Actually the right relation between $\alpha$ and $\kappa$ is
\begin{displaymath}
\alpha=\dfrac{1}{2}\dfrac{(3\kappa-8)(6-\kappa)}{2\kappa}.
\end{displaymath}
So the value of $\alpha$ that corresponds to $\kappa=4$ is $\frac{1}{2}$ and not $1$. The missing factor $\frac{1}{2}$ appears in the Lawler's work \cite{Lawler2009PartFuncSLELoopMes} (Proposition $2.1$). The error in
\cite{SheffieldWerner2012CLE} comes from an error in the article \cite{LawlerWerner2004ConformalLoopSoup} by Lawler and Werner. There the authors consider a Brownian loop soup in the half-plane and a continuous path cutting the half-plane, parametrised by the half-plane capacity. For such a path the half-plane capacity at time $t$ equals $2t$. It discovers progressively new Brownian loops and the authors map these loops conformally to the origin. In the Theorem $1$ they identify the processes of these conformally mapped Brownian loops to be a Poisson point process with intensity proportional to the Brownian bubble measure. In the identification of intensity there is a factor $2$ missing. Actually in the article \cite{LawlerWerner2004ConformalLoopSoup}, the Theorem $1$ is inconsistent with the Proposition $11$.

The problem of percolation by random walk loops was studied in \cite{LeJanLemaire2012LoopClusters}, \cite{ChangSapozhnikov2014PercLoops}, \cite{Lupu2014LoopsGFF} and \cite{Chang2015SuperCritLoops} in more general setting than dimension $2$. We will focus on the percolation by loops in $\mathcal{L}^{\mathsf{H}}_{\alpha}$. The probability of existence of an infinite cluster of loops follows a $0-1$ law and there can be at most one infinite cluster (\cite{Lupu2014LoopsGFF}). Moreover for $\alpha=\frac{1}{2}$ loops in $\mathcal{L}^{\mathsf{H}}_{\frac{1}{2}}$ do not percolate (\cite{Lupu2014LoopsGFF}). This result was obtained through a coupling with the massless Gaussian free field. By considering just the loops that go back and forth between two neighbouring vertices we get a lower bound on clusters of loops by clusters of an i.i.d. Bernoulli percolation. In particular this implies that for $\alpha$ large enough loops in $\mathcal{L}^{\mathsf{H}}_{\alpha}$ percolate. Hence as the parameter $\alpha$ increases there is a phase transition and a critical value $\alpha_{\ast}^{\mathsf{H}}\in[\frac{1}{2},+\infty)$ of the parameter. Using the results on the clusters of Brownian loops from \cite{SheffieldWerner2012CLE} and the approximation result from \cite{LawlerFerreras2007RWLoopSoup} we will show in section \ref{SecCcritLeq1} the following:

\begin{thm}
\label{ThmCcritLeq1}
For all $\alpha>\frac{1}{2}$ there is an infinite cluster of loops in $\mathcal{L}^{\mathsf{H}}_{\alpha}$. In particular $\alpha_{\ast}^{\mathsf{H}}=\frac{1}{2}$. Moreover, given $\alpha>\frac{1}{2}$, there is $n\in\mathbb{N}^{\ast}$ large enough, such that $\mathcal{L}^{\mathsf{H},\leq n}_{\alpha}$ percolates too.
\end{thm}

That is to say the critical intensity parameter for the two-dimensional Brownian loop soups and random walk loop soups is the same.

We will consider $1$-dependent edge percolations on $\mathsf{H}$, $(\omega(e))_{e~\text{edge}}$. By $1$-dependent percolation we mean that if two disjoint subsets of edges $E_{1}$ and $E_{2}$ are at graph distance at least $1$ then $(\omega(e))_{e\in E_{1}}$ and $(\omega(e))_{e\in E_{2}}$ are independent. According the results on locally dependent percolation by Liggett, Schonmann and Stacey in \cite{LiggettSchonmannStacey1997LocDepPerc}, for all $1$-dependent edge percolations on $\mathsf{H}$ with $p$ the probability of an edge to be open, there is an universal $\tilde{p}(p)\in[0,1)$ such that the $1$-dependent edge percolation contains an i.i.d. Bernoulli percolation with probability $\tilde{p}(p)$ of an edge to be open. Moreover the following constraint holds:
\begin{displaymath}
\lim_{p\rightarrow 1^{-}}\tilde{p}(p)=1.
\end{displaymath}

\section{Critical intensity parameter}
\label{SecCcritLeq1}

Let $\alpha,\delta >0$. Given $U$ an open subset of $\mathbb{H}$, we will denote by
$\mathcal{L}^{U,\geq\delta}_{\alpha}$ respectively $\mathcal{L}^{U\cap\mathsf{H},\geq\delta}_{\alpha}$ the subset of
$\mathcal{L}^{\mathbb{H}}_{\alpha}$ respectively $\mathcal{L}^{\mathsf{H}}_{\alpha}$ made of loops contained in $U$
and with diameter greater or equal to $\delta$. We will use the notations 
$\mathcal{L}^{U}_{\alpha}$ and $\mathcal{L}^{U\cap\mathsf{H}}_{\alpha}$ when there is a condition on the range but not
on the diameter.

Let $Q_{ext}$ and $Q_{int}$ be the following rectangles:
\begin{displaymath}
Q_{ext}:=(0,6)\times(0,3),\quad
Q_{int}:=(1,5)\times(1,2).
\end{displaymath}
We consider the subset of Brownian loops 
$\mathcal{L}^{Q_{ext},\geq\delta}_{\alpha}$, which is a.s. finite. We introduce the events 
$C_{1}(\mathcal{L}^{Q_{ext},\geq\delta}_{\alpha})$, 
$C_{2}(\mathcal{L}^{Q_{ext},\geq\delta}_{\alpha})$  and 
$C_{3}(\mathcal{L}^{Q_{ext},\geq\delta}_{\alpha})$ depending on the loops in $\mathcal{L}^{Q_{ext},\geq\delta}_{\alpha}$. 
The event $C_{1}(\mathcal{L}^{Q_{ext},\geq\delta}_{\alpha})$ will be satisfied if there is a cluster $K_{1}$ of loops in $\mathcal{L}^{Q_{int},\geq\delta}_{\alpha}$ such that in 
$\mathcal{L}^{(0,6)\times (1,2),\geq\delta}_{\alpha}$ there is a loop that intersects $K_{1}$ and $\lbrace 1\rbrace\times (1,2)$ and a loop that intersects $K_{1}$ and $\lbrace 5\rbrace\times (1,2)$. The two loops may be the same. $C_{2}(\mathcal{L}^{Q_{ext},\geq\delta}_{\alpha})$ will be satisfied if there is a cluster $K_{2}$ in 
$\mathcal{L}^{(1,2)^{2},\geq\delta}_{\alpha}$ such that in 
$\mathcal{L}^{(1,2)\times(0,3),\geq\delta}_{\alpha}$ there is a loop that intersects $K_{2}$ and $(1,2)\times\lbrace 1\rbrace$ and a loop that intersects $K_{2}$ and $(1,2)\times\lbrace 2\rbrace$. The event $C_{3}(\mathcal{L}^{Q_{ext},\geq\delta}_{\alpha})$ is similar to the event 
$C_{2}(\mathcal{L}^{Q_{ext},\geq\delta}_{\alpha})$ where the square $(1,2)^{2}$ is replaced by the square $(4,5)\times (1,2)$ and the rectangle
$(1,2)\times (0,3)$ by the rectangle $(4,5)\times(0,3)$. 
Next figure illustrates the event 
$\bigcap_{i=1}^{3}C_{i}(\mathcal{L}^{Q_{ext},\geq\delta}_{\alpha})$.

\begin{center}
\begin{pspicture}(0,-1.4)(9.3,5)
\psline[linestyle=dashed, linewidth=0.025](0,0)(9,0)
\psline[linestyle=dashed, linewidth=0.025](0,4.5)(9,4.5)
\psline[linestyle=dashed, linewidth=0.025](0,0)(0,4.5)
\psline[linestyle=dashed, linewidth=0.025](9,0)(9,4.5)
\psline[linewidth=0.025](1.5,1.5)(7.5,1.5)
\psline[linewidth=0.025](1.5,3)(7.5,3)
\psline[linewidth=0.025](1.5,1.5)(1.5,3)
\psline[linewidth=0.025](7.5,1.5)(7.5,3)
\psline[linestyle=dotted, linewidth=0.025](1.5,0)(1.5,4.5)
\psline[linestyle=dotted, linewidth=0.025](3,0)(3,4.5)
\psline[linestyle=dotted, linewidth=0.025](4.5,0)(4.5,4.5)
\psline[linestyle=dotted, linewidth=0.025](6,0)(6,4.5)
\psline[linestyle=dotted, linewidth=0.025](7.5,0)(7.5,4.5)
\psline[linestyle=dotted, linewidth=0.025](0,1.5)(9,1.5)
\psline[linestyle=dotted, linewidth=0.025](0,3)(9,3)
\pscurve[linestyle=dashed,linewidth=0.05]
(1.7,3.2)(2.1,3.5)(1.8,1)(2.6,1.3)(1.65,3.1)(1.7,3.2)
\pscurve[linestyle=dashed,linewidth=0.05]
(0.9,2.6)(0.9,2.5)(1.7,2.3)(2.7,2.1)(2.8,2.7)(1.5,2.7)(1.0,2.65)(0.9,2.6)
\pscurve[linewidth=0.05](1.8,1.6)(2.5,1.8)(2.2,2.5)(1.7,1.65)(1.8,1.6)
\pscurve[linewidth=0.05](2.5,2.75)(3.8,2.5)(4,1.7)(3,1.9)(4.7,2.1)(5.6,2.8)(5.8,2.4)(2.8,2.2)(2.4,2.7)(2.5,2.75)
\pscurve[linewidth=0.05](5.5,2.6)(6.5,2.25)(6.7,2)(5.8,1.7)(5.4,2.55)(5.5,2.6)
\pscurve[linewidth=0.05](6.2,2.6)(6.3,1.7)(7.3,2.2)(6.4,2.8)(6.2,2.6)
\pscurve[linestyle=dashed,linewidth=0.05]
(6.5,2.1)(6.8,1.95)(7.2,1.3)(6.3,2.15)(6.5,2.1)
\pscurve[linestyle=dashed,linewidth=0.05]
(7.1,2.8)(8,2.7)(7.2,1.9)(6.9,2.7)(7.1,2.8)
\pscurve[linestyle=dashed,linewidth=0.05]
(6.8,2.3)(7.3,2.7)(6.5,3.8)(6.7,4.1)(6.6,2.35)(6.8,2.3)
\rput(9.4,2.25){$Q_{ext}$}
\rput(4,3.2){$Q_{int}$}
\rput(4.65,-0.3){\begin{scriptsize}Fig.1: Illustration of the event 
$\bigcap_{i=1}^{3}C_{i}(\mathcal{L}^{Q_{ext},\geq\delta}_{\alpha})$.\end{scriptsize}}
\rput(4.65,-0.6){\begin{scriptsize} One should imagine that the smooth loops are actually Brownian.\end{scriptsize}}
\rput(4.65,-0.9){\begin{scriptsize} Only a set of loops that is sufficient for the event is represented.\end{scriptsize}}
\rput(4.65,-1.2){\begin{scriptsize} Full line loops stay inside $Q_{int}$. Dashed loops cross the boundary of $Q_{int}$.\end{scriptsize}}
\end{pspicture}
\end{center}

We will call the event $\bigcap_{i=1}^{3}C_{i}(\mathcal{L}^{Q_{ext},\geq\delta}_{\alpha})$ \textit{special crossing event with exterior rectangle $Q_{ext}$ and interior rectangle $Q_{int}$}. We will also consider translations, rotations and
rescaling of $Q_{ext}$ and $Q_{int}$ and deal with \textit{special crossing events} corresponding to the new rectangles. We are interested in the event 
$\bigcap_{i=1}^{3}C_{i}(\mathcal{L}^{Q_{ext},\geq\delta}_{\alpha})$
because then the loops in $\mathcal{L}^{Q_{ext},\geq\delta}_{\alpha}$ achieve the three crossings drawn on the figure $2$:

\begin{center}
\begin{pspicture}(0,-0.5)(9.3,5)
\psline[linestyle=dashed, linewidth=0.025](0,0)(9,0)
\psline[linestyle=dashed, linewidth=0.025](0,4.5)(9,4.5)
\psline[linestyle=dashed, linewidth=0.025](0,0)(0,4.5)
\psline[linestyle=dashed, linewidth=0.025](9,0)(9,4.5)
\psline[linewidth=0.025](1.5,1.5)(7.5,1.5)
\psline[linewidth=0.025](1.5,3)(7.5,3)
\psline[linewidth=0.025](1.5,1.5)(1.5,3)
\psline[linewidth=0.025](7.5,1.5)(7.5,3)
\psline[linestyle=dotted, linewidth=0.025](1.5,0)(1.5,4.5)
\psline[linestyle=dotted, linewidth=0.025](3,0)(3,4.5)
\psline[linestyle=dotted, linewidth=0.025](4.5,0)(4.5,4.5)
\psline[linestyle=dotted, linewidth=0.025](6,0)(6,4.5)
\psline[linestyle=dotted, linewidth=0.025](7.5,0)(7.5,4.5)
\psline[linestyle=dotted, linewidth=0.025](0,1.5)(9,1.5)
\psline[linestyle=dotted, linewidth=0.025](0,3)(9,3)
\psbezier[linewidth=0.05](1.2,2.6)(2.5,2.1)(7,2.8)(7.8,1.6)
\psbezier[linewidth=0.05](2.7,3.3)(2.7,2.9)(1.9,2.4)(2.3,1.2)
\psbezier[linewidth=0.05](6.7,3.3)(6.8,2.2)(6.6,1.7)(7.1,1.2)
\rput(9.4,2.25){$Q_{ext}$}
\rput(4,3.2){$Q_{int}$}
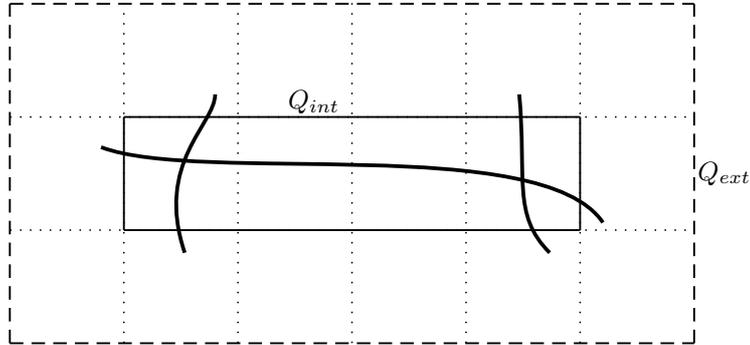
\rput(4.65,-0.3){\begin{scriptsize}Fig.2: The three crossings we are interested in.\end{scriptsize}}
\end{pspicture}
\end{center}

Next we show that if $\alpha>\frac{1}{2}$ and $\delta$ is small enough then the probability of $\bigcap_{i=1}^{3}C_{i}(\mathcal{L}^{Q_{ext},\geq\delta}_{\alpha})$ is close to $1$.

\begin{lemma}
\label{LemInstIntersec}
Let $Q$ be a rectangle of form $Q=(-a,a)\times (0,b)$. Let $\alpha>0$. Let $(B_{t})_{t\geq 0}$ be the standard Brownian motion on $\mathbb{C}$ started from $0$ and let $\mathcal{L}_{\alpha}^{Q}$ be a Poisson ensemble of loops independent from $B$. Then for all $\varepsilon>0$ there is 
$t\in(0,\varepsilon)$ such that $B$ at time $t$ intersects a loop in
$\mathcal{L}_{\alpha}^{Q}$. 
\end{lemma}

\begin{proof}
First we consider a loops soup in $\mathbb{H}$, $\mathcal{L}_{\alpha}^{\mathbb{H}}$, independent of $B$. Let 
\begin{displaymath}
T:=\inf\lbrace t>0\vert B_{t}~\text{is in the range of a loop in}~\mathcal{L}_{\alpha}^{\mathbb{H}}\rbrace.
\end{displaymath}
$T$ is a.s. finite. Indeed a loop in $\mathcal{L}_{\alpha}^{\mathbb{H}}$ delimits a domain with non-empty interior. Since the Brownian motion on $\mathbb{C}$ is recurrent, $B$ will visit this domain and thus intersect the loop. Let $\lambda >0$. The Poisson ensemble of loops $\mathcal{L}_{\alpha}^{\mathbb{H}}$ is invariant in law under the Brownian scaling
\begin{displaymath}
(\gamma(t))_{0\leq t\leq t_{\gamma}}\longmapsto
\lambda^{-\frac{1}{2}}(\gamma(\lambda t))_{0\leq t\leq \lambda^{-1} t_{\gamma}}.
\end{displaymath}
So does the Brownian motion $B$. Thus $\lambda T$ has the same law as $T$. It follows that $T=0$ a.s.

The set of loops $\mathcal{L}_{\alpha}^{\mathbb{H}}\setminus \mathcal{L}_{\alpha}^{Q}$ is at positive distance from $0$ thus $B$ cannot intersect it immediately. It follows that $B$ intersects immediately
$\mathcal{L}_{\alpha}^{Q}$.
\end{proof}

\begin{lemma}
\label{LemCrossing}
Let $a,\alpha>0$. There is a.s. a loop in $\mathcal{L}_{\alpha}^{(-a,a)^{2}}$ that intersects the real line $\mathbb{R}$.
\end{lemma}

\begin{proof}
Let $\mathcal{L}_{\alpha}^{(n)}$ be the subset of 
$\mathcal{L}_{\alpha}^{(-a,a)^{2}}$ made of loops $\gamma$ of duration $t_{\gamma}$ comprised between $2^{-n-1}$ and 
$2^{-n}$. The family $(\mathcal{L}_{\alpha}^{(n)})_{n\geq 0}$ is independent. By Brownian scaling, the probability that a loop in $\mathcal{L}_{\alpha}^{(n)}$ intersects $\mathbb{R}$ is the same as a loop in 
$\mathcal{L}_{\alpha}^{(-a 2^{\sfrac{n}{2}},a 2^{\sfrac{n}{2}})^{2}}$ of duration comprised between $\frac{1}{2}$ and $1$ intersects $\mathbb{R}$. This is at least as big as the similar probability for $\mathcal{L}_{\alpha}^{(0)}$. Since the latter probability is non-zero, the intersection events occurs a.s. for infinitely many of $\mathcal{L}_{\alpha}^{(n)}$.
\end{proof}

\begin{lemma}
\label{LemMeetCluster}
Let $a,\alpha>0$. There is a.s. a loop in $\mathcal{L}_{\alpha}^{(-a,a)^{2}}$ that intersects the real line $\mathbb{R}$ and a loop in $\mathcal{L}_{\alpha}^{(-a,a)\times (0,a)}$.
\end{lemma}

\begin{proof}
Consider the subset of $\mathcal{L}_{\alpha}^{(-a,a)^{2}}$ made of loops intersecting $\mathbb{R}$. It is non empty according the lemma \ref{LemCrossing}. Moreover it is independent of
$\mathcal{L}_{\alpha}^{(-a,a)\times (0,a)}$. The law of a Brownian loop that intersects $\mathbb{R}$ is locally, near the point of intersection, absolutely continuous with respect to the law of a Brownian motion started from there. Applying lemma \ref{LemInstIntersec}, we get that it intersects a.s. a loop in $\mathcal{L}_{\alpha}^{(-a,a)\times (0,a)}$.
\end{proof}

\begin{lemma}
\label{LemConvProbCross1}
Let $\alpha>\frac{1}{2}$. Then
\begin{displaymath}
\lim_{\delta \rightarrow 0^{+}}\mathbb{P}\Big(\bigcap_{i=1}^{3}C_{i}(\mathcal{L}^{Q_{ext},\geq\delta}_{\alpha})\Big)=1.
\end{displaymath}
\end{lemma}

\begin{proof}
It is enough to show that the probability of each of the 
$C_{i}(\mathcal{L}^{Q_{ext},\geq\delta}_{\alpha})$ converges to $1$ as $\delta$ tends to $0$. Since the three cases are very similar, we will do the proof only for $C_{1}(\mathcal{L}^{Q_{ext},\geq\delta}_{\alpha})$.  According to lemma \ref{LemMeetCluster} there is a loop $\gamma$ in $\mathcal{L}^{(0,6)\times(1,2)}_{\alpha}$ that intersects 
$\lbrace 1\rbrace\times (1,2)$ and a loop $\gamma'$ in
$\mathcal{L}^{Q_{int}}_{\alpha}$. Similarly there is a loop $\tilde{\gamma}$ in $\mathcal{L}^{(0,6)\times(1,2)}_{\alpha}$ that intersects 
$\lbrace 5\rbrace\times (1,2)$ and a loop $\tilde{\gamma}'$ in
$\mathcal{L}^{Q_{int}}_{\alpha}$. Since $\alpha>\frac{1}{2}$, $\gamma'$ and $\tilde{\gamma}'$
belong to the same cluster in $\mathcal{L}^{Q_{int}}_{\alpha}$ (\cite{SheffieldWerner2012CLE}). Thus there is a chain of loops $(\gamma_{0},\dots,\gamma_{n})$ in $\mathcal{L}^{Q_{int}}_{\alpha}$, with $\gamma_{0}=\gamma'$ and
$\gamma_{n}=\tilde{\gamma}'$, joining $\gamma'$ and $\tilde{\gamma}'$. If $\delta$ is the minimum of diameters of $(\gamma_{0},\dots,\gamma_{n})$ and $\gamma$ and $\tilde{\gamma}$ then 
$C_{1}(\mathcal{L}^{Q_{ext},\geq\delta}_{\alpha})$ is satisfied. Let $\bar{\delta}$ be maximal value of $\delta$ such that $C_{1}(\mathcal{L}^{Q_{ext},\geq\delta}_{\alpha})$ is satisfied. $\bar{\delta}$ is a well defined random variable with values in $(0,+\infty)$. Then
\begin{displaymath}
\lim_{\delta \rightarrow 0^{+}}\mathbb{P}
(C_{1}(\mathcal{L}^{Q_{ext},\geq\delta}_{\alpha}))=
\lim_{\delta \rightarrow 0^{+}}\mathbb{P}(\delta\leq \bar{\delta})=1.
\end{displaymath}
\end{proof}

Next we recall the  result on approximation of Brownian loops by random walk loops from \cite{LawlerFerreras2007RWLoopSoup}. Let $N\in\mathbb{N}^{\ast}$. We consider the discrete loops $\gamma$ on $\mathbb{Z}\times\mathbb{N}^{\ast}$. We define on these loops a map $\Phi_{N}$ to continuous loops on $\mathbb{H}$. Given
$\gamma$ a discrete loop and $(z_{0},\dots,z_{n-1},z_{0})$ the sequence of the vertices it visits, the continuous loop
$\Phi_{N}\gamma$ satisfies:
\begin{itemize}
\item the duration of $\Phi_{N}\gamma$ is $\frac{n}{2N^{2}}$;
\item for $j\in \lbrace 0,\dots,n-1\rbrace$, $\Phi_{N}\gamma(\frac{j}{2N^{2}})=\frac{z_{j}}{N}$;
\item $\Phi_{N}\gamma(\frac{n}{2N^{2}})=\Phi_{N}\gamma(0)=\frac{z_{0}}{N}$;
\item between the times $\frac{j}{2N^{2}}$, $j\in \lbrace 0,\dots,n\rbrace$, $\Phi_{N}\gamma$ interpolates linearly.
\end{itemize}
The number of jumps $n$ of a discrete loop $\gamma$ will be denoted $s_{\gamma}$.
The life-time of a continuous loop $\tilde{\gamma}$ will be denoted by $t_{\tilde{\gamma}}$.
Let $\theta\in(\frac{2}{3},2)$ and $r\geq 1$. There is a coupling between $\mathcal{L}_{\alpha}^{\mathsf{H}}$ and $\mathcal{L}_{\alpha}^{\mathbb{H}}$ such that except on an event of probability at most $cste\cdot (\alpha+1)r^{2}N^{2-3\theta}$ there is a one to one correspondence between the two sets
\begin{itemize}
\item $\lbrace\gamma\in\mathcal{L}_{\alpha}^{\mathsf{H}}\vert s_{\gamma}>2N^{\theta},
\vert\gamma(0)\vert<Nr\rbrace$,
\item $\lbrace \tilde{\gamma}\in \mathcal{L}_{\alpha}^{\mathbb{H}}\vert t_{\tilde{\gamma}}>N^{\theta-2},
\vert\tilde{\gamma}(0)\vert<r\rbrace$,
\end{itemize}
such that given a discrete loop $\gamma$ and the continuous loop $\tilde{\gamma}$ corresponding to it,
\begin{displaymath}
\Big\vert\dfrac{s_{\gamma}}{2N^{2}}-t_{\tilde{\gamma}}\Big\vert\leq\dfrac{5}{8}N^{-2},\qquad
\sup_{0\leq u\leq 1}\Big\vert\Phi_{N}\gamma\Big(u\frac{s_{\gamma}}{2N^{2}}\Big)-\tilde{\gamma}
(u t_{\tilde{\gamma}})\Big\vert\leq
cste\cdot N^{-1}\log(N).
\end{displaymath}
Next we state without proof a lemma that follows immediately from this approximation.

\begin{lemma}
\label{LemApproxRectangle}
Let $\alpha>0$ and $\delta>0$. As $N$ tends to $+\infty$ the random set of interpolating continuous loops
\begin{displaymath}
\big\lbrace \Phi_{N}\gamma\vert \gamma\in\mathcal{L}_{\alpha}^{N Q_{ext}\cap\mathsf{H},\geq N\delta}\big\rbrace
\end{displaymath}
converges in law to the set of Brownian loops $\mathcal{L}_{\alpha}^{Q_{ext},\geq \delta}$.
\end{lemma}

We need to show that the above convergence for the uniform norm also implies a convergence of the intersection relations, that is to say that
\begin{displaymath}
\lbrace (\gamma,\gamma')\vert\gamma,\gamma'\in\mathcal{L}_{\alpha}^{N Q_{ext}\cap\mathsf{H},\geq N\delta},
\gamma~\text{intersects}~\gamma'\rbrace
\end{displaymath}
converges in law to
\begin{displaymath}
\lbrace (\tilde{\gamma},\tilde{\gamma}')\vert\tilde{\gamma},\tilde{\gamma}'\in
\mathcal{L}_{\alpha}^{Q_{ext},\geq \delta},\tilde{\gamma}~\text{intersects}~\tilde{\gamma}'\rbrace.
\end{displaymath}

Let $j\in\mathbb{N}$. Let $\gamma$ be a continuous path on $\mathbb{C}$ (not necessarily a loop) of lifetime $t_{\gamma}$. For $r>0$ let
\begin{displaymath}
T_{r}(\gamma):=\inf\lbrace s>0\vert \vert\gamma(s)\vert\geq r\rbrace \in (0,+\infty].
\end{displaymath}
If $T_{r}(\gamma)<+\infty$ let
\begin{displaymath}
e^{i\omega_{r}}:=\dfrac{\gamma(T_{r}(\gamma))}{r}.
\end{displaymath}
Let $I_{j}$ be the real interval
\begin{displaymath}
I_{j}:=\Big(\frac{7}{12}2^{-j},\frac{9}{12}2^{-j}\Big).
\end{displaymath}
For $0<r_{1}<r_{2}$ let $\mathcal{A}(r_{1},r_{2})$ be the annulus
\begin{displaymath}
\mathcal{A}(r_{1},r_{2}):=\lbrace z\in\mathbb{C}\vert r_{1}<\vert z\vert < r_{2}\rbrace.
\end{displaymath}
For $r>0$ let $HD(r)$ be the half-disc
\begin{displaymath}
HD(r):=B(0,r)\cap\lbrace z\in\mathbb{C}\vert \Re(z)>0\rbrace.
\end{displaymath}
We will say that the path $\gamma$ satisfies the condition $\mathcal{C}_{j}$ if
\begin{itemize}
\item $T_{\frac{11}{12}2^{-j}}(\gamma)<+\infty$,
\item after time $T_{\frac{11}{12}2^{-j}}(\gamma)<+\infty$, $\gamma$ hits 
$e^{i(\omega_{2^{-j-1}}+\frac{\pi}{2})}I_{j}$ at a time $\tilde{t}_{j}$ before hitting the circle $S(0,2^{-j})$,
\item on the time interval $(T_{2^{-j-1}}(\gamma),\tilde{t}_{j})$ $\gamma$ stays in the half-disc 
$e^{i\omega_{2^{-j-1}}}HD(2^{-j})$,
\item from time $\tilde{t}_{j}$ the path $\gamma$ stays in the annulus 
$\mathcal{A}(\frac{7}{12}2^{-j},\frac{9}{12}2^{-j})$ until surrounding the disc $B(0,\frac{7}{12}2^{-j})$ once clockwise and hitting $e^{i(\omega_{2^{-j-1}}+\pi)}I_{j}$.
\end{itemize}
Figure $3$ illustrates a path satisfying the condition $\mathcal{C}_{j}$. If this condition is satisfied than $\gamma$ disconnects the disc $B(0,\frac{7}{12}2^{-j})$ from infinity. Moreover if one perturbs $\gamma$ by any continuous function 
$f:[0,t_{\gamma}]\rightarrow\mathbb{C}$ such that 
$\Vert f\Vert_{\infty}\leq \frac{1}{12}2^{-j}$ then the path $(\gamma(s)+f(s))_{0\leq s\leq t_{\gamma}}$ disconnects the disc $B(0,2^{-j-1})$ from infinity. The disconnection is made inside the annulus $\mathcal{A}(2^{-j-1},2^{-j})$.

\begin{center}
\begin{pspicture}(0,-0.5)(6,6)
\pscircle[linestyle=dashed](3,3){3}
\pscircle[linestyle=dashed](3,3){1.5}
\pscircle[linestyle=dashed](3,3){1.75}
\pscircle[linestyle=dashed](3,3){2.25}
\pscircle[linestyle=dashed](3,3){2.75}
\psline[linestyle=dashed](0.764,1)(5.236,5)
\psline[linestyle=dashed](0.764,5)(5.236,1)
\pscurve[linewidth=0.05](2.9,2.9)(3,3)(3.5,3)(4,3.7)(4.1,4)(3.9,4.6)
(4,5.3)(3.7,5.8)(3.5,5.5)(2.4,4.8)(1.8,4.6)(1.5,4.3)
(1.1,3.9)(1.1,2)(1.4,1.6)(2.8,0.9)(4,1.4)(4.8,2.2)
(4.9,3.7)(4.4,4.4)(3.8,4.9)(2,4.9)(1.4,4)(1.2,3.5)
(1.1,2.7)(1.1,2.4)(1.3,1.9)(1.9,1.5)
\psdot[dotsize=3pt 2](4.1,3.97)
\rput(3.35,3.97){\begin{small}
$\frac{e^{i\omega_{2^{-j-1}}}}{2^{j+1}}$
\end{small}}
\psdot[dotsize=3pt 2](4.1,3.97)
\rput(3.2,5.5){$\gamma$}
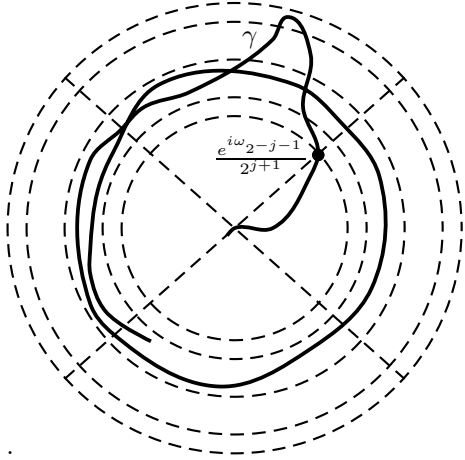
\rput(3,-0.3){\begin{scriptsize}Fig.3: Representation of a path $\gamma$ satisfying the condition $\mathcal{C}_{j}$\end{scriptsize}}.
\end{pspicture}
\end{center}

\begin{lemma}
\label{LemInfinityDisconnect}
Let $(B_{t})_{0\leq t\leq T}$ be a standard Brownian path on $\mathbb{C}$ starting from $0$. Then almost surely it satisfies the condition $\mathcal{C}_{j}$ for infinitely many values of $j\in\mathbb{N}$.
\end{lemma}

\begin{proof}
Let $\widetilde{B}$ be the Brownian path $B$ continued for $t\in(0,+\infty)$. The events $"\widetilde{B}~\text{satisfies the condition}~\mathcal{C}_{j}"$ are i.i.d. Indeed such an event is rotation invariant and depends only on 
$\widetilde{B}$ on the time interval $(T_{2^{-j-1}}(\widetilde{B}),T_{2^{-j}}(\widetilde{B}))$. Moreover the probability of such an event is non-zero. Thus $\widetilde{B}$ satisfies the condition $\mathcal{C}_{j}$ for infinitely many values of $j\in\mathbb{N}$. Since
\begin{displaymath}
\lim_{j\rightarrow +\infty}T_{2^{-j}}(\widetilde{B})=0,
\end{displaymath}
so does $B$.
\end{proof}

\begin{lemma}
\label{LemIntersections}
Let $z_{1},z_{2}\in\mathbb{C}$ and $t_{1},t_{2}>0$. Let $(b^{(1)}_{s})_{0\leq s\leq t_{1}}$ and
$(b^{(2)}_{s})_{0\leq s\leq t_{2}}$ be two independent standard Brownian bridges from $z_{1}$ to $z_{1}$ and 
$z_{2}$ to $z_{2}$ respectively. On the event that $b^{(1)}$ intersects $b^{(2)}$ there is a.s. $\varepsilon>0$ such that
for all continuous functions $f_{1}:[0,t_{1}]\rightarrow \mathbb{C}$ and $f_{2}:[0,t_{2}]\rightarrow \mathbb{C}$ of infinity norm $\Vert f_{i}\Vert_{\infty}\leq\varepsilon$, $(b^{(1)}_{s}+f_{1}(s))_{0\leq s\leq t_{1}}$ intersects
$(b^{(2)}_{s}+f_{2}(s))_{0\leq s\leq t_{2}}$.
\end{lemma}

\begin{proof}
Let $T^{(1)}_{2}$ be the first time $b^{(1)}$ hits the range of $b^{(2)}$. If the two path do not intersect each other 
$T^{(1)}_{2}=+\infty$. On the event $T^{(1)}_{2}<+\infty$ the conditional law of 
$(b^{(1)}_{T^{(1)}_{2}+s}-b^{(1)}_{T^{(1)}_{2}})_{0\leq s\leq t_{1}-T^{(1)}_{2}-\varepsilon}$ ($\varepsilon>0$ a small constant) given the value $T^{(1)}_{2}$ is absolutely continuous with respect the law of a Brownian path starting from $0$. From lemma \ref{LemInfinityDisconnect} follows that
the path $(b^{(1)}_{T^{(1)}_{2}+s}-b^{(1)}_{T^{(1)}_{2}})_{0\leq s\leq t_{1}-T^{(1)}_{2}}$ satisfies the condition $\mathcal{C}_{j}$ for infinitely many values of $j\in\mathbb{N}$. Let 
\begin{multline*}
\tilde{\jmath}:=\max\Big\lbrace j\in\mathbb{N}\vert
(b^{(1)}_{T^{(1)}_{2}+s}-b^{(1)}_{T^{(1)}_{2}})_{0\leq s\leq t_{1}-T^{(1)}_{2}}~
\text{satisfies the condition}~\mathcal{C}_{j}\\
\text{and}~\exists s\in [0,t_{2}], \vert b^{(2)}_{s}- b^{(2)}_{T^{(1)}_{2}}\vert\geq \frac{13}{12}2^{-j}\Big\rbrace.
\end{multline*}
$\tilde{\jmath}$ is a r.v. defined on the event where $b^{(1)}$ and $b^{(2)}$ intersect. If $f_{1}$ and $f_{2}$ are such that $\Vert f_{i}\Vert\leq\frac{1}{12}2^{-\tilde{\jmath}}$ then the path $b^{(1)}+f_{1}$ disconnects the disc 
$B(b^{(1)}_{T^{(1)}_{2}},2^{-\tilde{\jmath}-1})$ from infinity inside the annulus  
$b^{(1)}_{T^{(1)}_{2}}+\mathcal{A}(2^{-\tilde{\jmath}-1},2^{-\tilde{\jmath}})$
and the path $b^{(2)}+f_{2}$ crosses from the circle $S(b^{(1)}_{T^{(1)}_{2}},2^{-\tilde{\jmath}-1})$ to the circle $S(b^{(1)}_{T^{(1)}_{2}},2^{-\tilde{\jmath}})$, so the two must intersect.
\end{proof}

Observe that two discrete loops $\gamma$ and $\gamma'$ intersect each other if and only if the continuous loops
$\Phi_{N}\gamma$ and $\Phi_{N}\gamma'$ do. From lemmas \ref{LemApproxRectangle} and \ref{LemIntersections} follows:

\begin{corollary}
\label{CorIntersections}
Let $\alpha>0$ and $\delta>0$. As $N$ tends to $+\infty$ the random set of interpolating continuous loops
\begin{displaymath}
\big\lbrace \Phi_{N}\gamma\vert \gamma\in\mathcal{L}_{\alpha}^{NQ_{ext}\cap\mathsf{H},\geq N\delta}\big\rbrace
\end{displaymath}
jointly with the intersection relations
\begin{displaymath}
\lbrace (\gamma,\gamma')\vert\gamma,\gamma'\in\mathcal{L}_{\alpha}^{NQ_{ext}\cap\mathsf{H},\geq N\delta},
\gamma~\text{intersects}~\gamma'\rbrace
\end{displaymath}
converges in law to the set of Brownian loops $\mathcal{L}_{\alpha}^{Q_{ext},\geq \delta}$ jointly with the intersection relations
\begin{displaymath}
\lbrace (\tilde{\gamma},\tilde{\gamma}')\vert\tilde{\gamma},\tilde{\gamma}'\in
\mathcal{L}_{\alpha}^{Q_{ext},\geq \delta},\tilde{\gamma}~\text{intersects}~\tilde{\gamma}'\rbrace.
\end{displaymath}
\end{corollary}

We consider the scaled up rectangle $N Q_{ext}$ and $N Q_{int}$. The next lemma deals with the probability that the discrete loops $\mathcal{L}_{\alpha}^{NQ_{ext}\cap\mathsf{H}}$ realise the \textit{special crossing event with exterior rectangle $N Q_{ext}$ and interior rectangle $N Q_{int}$}. See figures $1$ and $2$ and consider that $Q_{ext}$ is replaced by $N Q_{ext}$, $Q_{int}$ by $N Q_{int}$ and $\mathcal{L}_{\alpha}^{Q_{ext},\geq\delta}$ by
$\mathcal{L}_{\alpha}^{NQ_{ext}\cap\mathsf{H}}$.

\begin{lemma}
\label{LemConvSpecCross}
Let $\alpha>\frac{1}{2}$.
As $N$ tends to $+\infty$, the probability that the loops $\mathcal{L}_{\alpha}^{NQ_{ext}\cap\mathsf{H}}$ realise a \textit{special crossing event with exterior rectangle $N Q_{ext}$ and interior rectangle $N Q_{int}$}
converges to $1$.
\end{lemma}

\begin{proof}
Let $\delta>0$.
The probability that the loops $\mathcal{L}_{\alpha}^{NQ_{ext}\cap\mathsf{H}}$ realise the \textit{special crossing event with exterior rectangle $N Q_{ext}$ and interior rectangle $N Q_{int}$} is at least as large as the probability that the loops $\mathcal{L}_{\alpha}^{NQ_{ext}\cap\mathsf{H},\geq N\delta}$ realise the \textit{special crossing event} with the same interior and exterior rectangle. From the corollary \ref{CorIntersections} follows that the latter probability converges as $N\rightarrow +\infty$ to
\begin{displaymath}
\mathbb{P}\left(\bigcap_{i=1}^{3}C_{i}(\mathcal{L}^{Q_{ext},\geq\delta}_{\alpha})\right).
\end{displaymath}
We conclude by applying the lemma \ref{LemConvProbCross1}.
\end{proof}

To conclude that for $\alpha>\frac{1}{2}$, $\mathcal{L}_{\alpha}^{\mathsf{H}}$ has an infinite cluster we will use a block percolation construction that will combine \textit{special crossing events}. 

\begin{proofof1}
From \cite{Lupu2014LoopsGFF} we know already that $\alpha_{\ast}^{\mathsf{H}}\leq\frac{1}{2}$. We need to show that
for $\alpha>\frac{1}{2}$, $\mathcal{L}_{\alpha}^{\mathsf{H}}$ has an infinite cluster.

Let $\alpha>\frac{1}{2}$ and $N\geq 1$. We consider a dependent edge percolation, 
denoted $(\omega^{N}(e))_{e~\text{edge of}~\mathsf{H}}$, on the discrete half plane $\mathsf{H}$. 
If $e$ is an edge of form 
$\lbrace(j,k),(j+1,k)\rbrace$, $k\geq 1$, then $\omega^{N}(e)=1$ (open edge) if 
$\mathcal{L}_{\alpha}^{(N Q_{int}+3Nj+i3Nk)\cap\mathsf{H}}$
achieves a \textit{special crossing event with exterior rectangle $N Q_{ext}+3Nj+i3Nk$ and interior rectangle 
$N Q_{int}+3Nj+i3Nk$}. If $e$ is an edge of form 
$\lbrace(j,k),(j,k+1)\rbrace$, $k\geq 1$, then 
$\omega^{N}(e)=1$ if $\mathcal{L}_{\alpha}^{(iN Q_{int}+3Nj+i3Nk)\cap\mathsf{H}}$
achieves a \textit{special crossing event with exterior rectangle $iN Q_{ext}+3Nj+i3Nk$ and interior rectangle 
$iN Q_{int}+3Nj+i3Nk$}, where the multiplication by $i$ means rotation by $+\frac{\pi}{2}$. $\omega^{N}$ is a $1$-dependent edge percolation: if two disjoint subsets of edges $E_{1}$ and $E_{2}$ are such that no edge is adjacent to both $E_{1}$ and $E_{2}$, then $(\omega^{N}(e))_{e\in E_{1}}$ and $(\omega^{N}(e))_{e\in E_{2}}$ are independent. This is due to the fact that the subsets of loops involved in the definition of \textit{special crossing events} for edges in $E_{1}$ and and edges in $E_{2}$ are disjoint.  To an open path in $\omega^{N}$ corresponds a cluster of $\mathcal{L}_{\alpha}^{\mathsf{H}}$ whose loops form crossings of related interior rectangles. Thus if $\omega^{N}$ has an unbounded cluster, then so does $\mathcal{L}_{\alpha}^{\mathsf{H}}$. See next picture.

\begin{center}
\begin{pspicture}(0,-0.5)(9,9)
\psline[linestyle=dashed, linewidth=0.025](0,0)(9,0)
\psline[linestyle=dashed, linewidth=0.025](0,0)(0,9)
\psline[linestyle=dashed, linewidth=0.025](9,0)(9,9)
\psline[linestyle=dashed, linewidth=0.025](0,9)(9,9)
\psline[linestyle=dashed, linewidth=0.025](0,3)(9,3)
\psline[linestyle=dashed, linewidth=0.025](3,0)(3,9)
\psline[linestyle=dashed, linewidth=0.025](6,0)(6,9)
\psline[linestyle=dashed, linewidth=0.025](0,6)(9,6)
\psline[linewidth=0.025](0,1)(9,1)
\psline[linewidth=0.025](0,2)(9,2)
\psline[linewidth=0.025](0,4)(9,4)
\psline[linewidth=0.025](0,5)(9,5)
\psline[linewidth=0.025](0,7)(9,7)
\psline[linewidth=0.025](0,8)(9,8)
\psline[linewidth=0.025](1,0)(1,9)
\psline[linewidth=0.025](2,0)(2,9)
\psline[linewidth=0.025](4,0)(4,9)
\psline[linewidth=0.025](5,0)(5,9)
\psline[linewidth=0.025](7,0)(7,9)
\psline[linewidth=0.025](8,0)(8,9)
\psbezier[linewidth=0.05](0.9,4.3)(2.4,4.8)(3.8,4.2)(5.1,4.6)
\psbezier[linewidth=0.05](1.3,3.9)(1.7,4.3)(1.2,4.7)(1.9,5.1)
\psbezier[linewidth=0.05](4.4,3.9)(4.6,4.3)(4.2,4.7)(4.3,5.1)
\psbezier[linewidth=0.05](3.9,4.1)(5.4,4.9)(7.8,4.6)(8.1,4.3)
\psbezier[linewidth=0.05](4.5,3.9)(4.8,4.3)(4.9,4.7)(4.5,5.1)
\psbezier[linewidth=0.05](7.2,3.9)(7.6,4.3)(7.6,4.7)(7.9,5.1)
\psbezier[linewidth=0.05](7.4,3.9)(7.4,5.5)(7.6,7)(7.9,8.1)
\psbezier[linewidth=0.05](6.9,4.3)(7.3,4.7)(7.6,4.7)(8.1,4.8)
\psbezier[linewidth=0.05](6.9,7.3)(7.3,7.1)(7.6,7.7)(8.1,7.8)
\psbezier[linewidth=0.05](3.9,7.5)(5.4,7.1)(7.8,7.9)(8.1,7.3)
\psbezier[linewidth=0.05](7.2,6.9)(7.3,7.3)(7.6,7.6)(7.4,8.1)
\psbezier[linewidth=0.05](4.4,6.9)(4.3,7.3)(4.6,7.6)(4.5,8.1)
\psbezier[linewidth=0.05](4.9,0.9)(4.3,1.3)(4.6,1.6)(4.5,2.1)
\psbezier[linewidth=0.05](1.8,0.9)(1.3,1.3)(1.6,1.6)(1.9,2.1)
\psbezier[linewidth=0.05](0.9,1.7)(2.4,1.3)(3.8,1.2)(5.1,1.6)
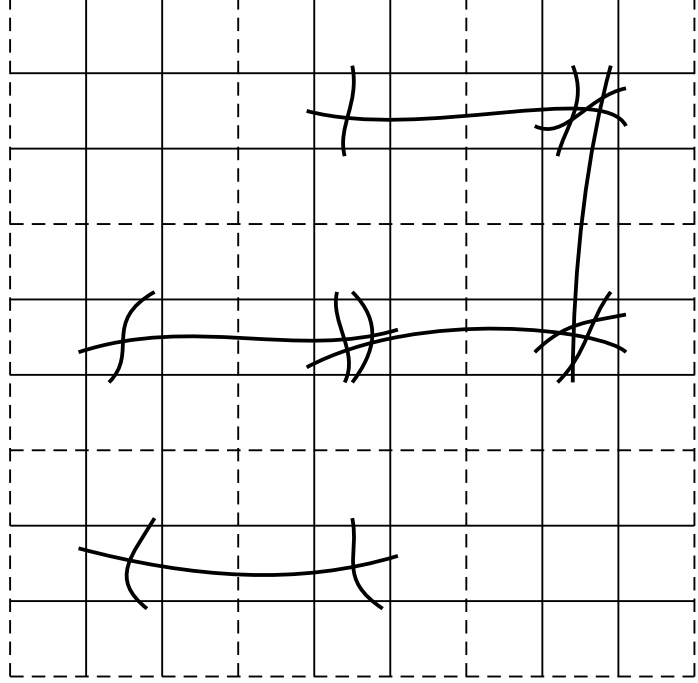
\rput(4.5,-0.3){\begin{scriptsize}Fig.4: Crossings achieved by subsets of loops in $\mathcal{L}_{\alpha}^{\mathsf{H}}$, corresponding to five open edges in $\omega^{N}$.\end{scriptsize}}
\end{pspicture}
\end{center}

The probability $\mathbb{P}(\omega^{N}(e)=1)$ is uniform and we will denote it $p_{N}$. According to the lemma
\ref{LemConvSpecCross}
\begin{displaymath}
\lim_{N\rightarrow +\infty}p_{N}=1.
\end{displaymath}
Thus for $N$ large enough $\tilde{p}(p_{N})>\frac{1}{2}$. $\frac{1}{2}$ is the critical probability for the i.i.d. Bernoulli edge percolation on $\mathsf{H}$. So for $N$ large enough $\omega^{N}$ contains a supercritical
i.i.d. Bernoulli edge percolation and percolates itself. Thus $\mathcal{L}_{\alpha}^{\mathsf{H}}$ percolates too.
Actually, since our construction only uses loops of diameter less or equal to $6N$, we have also percolation for
$\mathcal{L}_{\alpha}^{\mathsf{H},\leq 6N}$.
\end{proofof1}

\section*{Acknowledgements}

The author would like to thank Artem Sapozhnikov and Yinshan Chang for their comments on previous versions of this paper.

\bibliographystyle{plain}
\bibliography{tituslupuabr}

\end{document}